\documentclass[a4paper,12pt]{article}

\usepackage{color}
\usepackage{graphicx}
\usepackage{graphics}
\usepackage {amsfonts}
\usepackage {amsmath}
\usepackage {amsthm}
\usepackage {amssymb}
\usepackage {framed}
\usepackage {amsxtra}
\usepackage {enumerate}
\usepackage {fancyhdr}



\title{Antiassociative groupoids}
\author{M. Braitt, D. Hobby, and D. Silberger}

\newcommand{\term}{\mathop{\rm term}}

\newcommand{\path}{\mathop{\rm path}}
\newcommand{\Paths}{\mathop{\rm Paths}}

\newcommand{\Ltwo}{\mathop{\rm L_2}}
\newcommand{\Rthree}{\mathop{\rm R_3}}

\theoremstyle{definition}
\newtheorem{df}{Definition} [section]

\theoremstyle{plain}
\newtheorem{theorem}[df]{Theorem}

\newtheorem{lemma}[df]{Lemma}

\newtheorem{conjecture}[df]{Conjecture}

\newtheorem{question}[df]{Question}

\begin{document}

\centerline{\Large\bf Antiassociative groupoids}\vspace{2em}

\centerline{$^1$Milton Braitt}

\centerline{$^2$David Hobby}

\centerline{$^2$Donald Silberger}

\begin{abstract}
Given a groupoid \ $\langle G, \star \rangle$, \ and
 \ $k \geq 3$, \ we say that $G$ is {\em antiassociative} 
 iff for all \ $x_1, x_2, x_3 \in G$, \ 
 $(x_1 \star x_2) \star x_3$ \ and \ $x_1 \star (x_2 \star x_3)$ \ 
 are never equal.  Generalizing this, \ $\langle G, \star \rangle$ \ 
is $k$-antiassociative iff for all 
\ $x_1, x_2, \ldots x_k \in G$, \ any two distinct expressions
made by putting parentheses in
 \ $x_1 \star x_2 \star x_3 \star \cdots x_k$ \ are never equal.
 
We prove that for every $k \geq 3$, there exist finite groupoids
that are $k$-antiassociative.  We then generalize this, investigating when other pairs of groupoid terms can be made never equal.

\end{abstract}

\begin{section}{Introduction}

Around fifteen years ago, the second two authors 
started to investigate finite groupoids which were 
antiassociative.  Instead of obeying the associative law that \ 
$(x_1 \star x_2) \star x_3$ \ and \ $x_1 \star (x_2 \star x_3)$ \ 
are always equal, a groupoid is {\em antiassociative} iff \ 
$(x_1 \star x_2) \star x_3$ \ and \ $x_1 \star (x_2 \star x_3)$ \ 
are never equal.  This is a natural change to make
to the associative law.  

We were aided by a program written by Ming Lei Wu, which went 
through all the $4^{16}$ possible $4$-element groupoids and 
returned a list of 421,560 which were antiassociative.  About 
97\% of these antiassociative groupoids were what we called
``deranged'', and turned out to be constructible in the following way.

Let $G$ be any set with $2$ or more elements.  First pick a 
function \ $f : G \rightarrow G$ \ with the 
property that \ $f(x) \neq x$ \ for all $x$ 
(the ``derangement'').
Then define the binary operation on $G$ by  \ 
$x \star y = f(x)$, \ or alternatively, by \ 
$x \star y = f(y)$.  This makes $\langle G, \star \rangle$
a {\em deranged} groupoid. \  
When $x \star y = f(x)$, we have \ 
$(x_1 \star x_2) \star x_3 = f(x_1) \star x_3 
= f(f(x_1)) \neq f(x_1) = x_1 \star (x_2 \star x_3)$,
showing $\langle G,\star \rangle$ is antiassociative.
If $x \star y = f(y)$, the proof is similar.

Of the remaining 3\% of the antiassociative groupoids
found by the program, almost all had $\star$ tables which were
within a few entries of the table of one of the deranged
groupoids.  But beyond that, we found few patterns 
in their construction.  We conjecture that a
similar situation holds for the examples we give
in this paper.  They probably will not be unique, since
it will sometimes be possible to modify them slightly
in a haphazard way.

Before moving on to $k$-antiassociative groupoids, we
will invest in some definitions.  Using terminology from
universal algebra (see \cite{BurrisSankappanavar}), an 
{\em algebra} is a set with some number of (finitary) operations on it.
A {\em term} of an algebra is any expression on a finite
number of variables that can be made by composing the
(basic) operations of the algebra.  We will use the same
notation both for terms as formal expressions and for the
resulting functions on an algebra, since the distinction
should be clear from context.  This paper will 
focus on groupoids, which are algebras with a single 
binary operation.  We believe that many of our techniques
can be used for algebras with multiple operations of any arity,
but will not pursue this avenue here.

An {\em ordered term} on the variables \ 
$x_j, x_{j+1}, x_{j+2}, \ldots x_{j+k-1}$ is a $k$-ary
term where each variable appears once, in order of
their indices.  For clarity, we give an inductive
definition.  Any single variable $x_j$ is a $1$-ary
ordered term.  Now suppose that $f$ is an $m$-ary 
basic operation and that \ $t_1, \ldots t_m$ \ are 
ordered terms on the variables \ 
$x_j, x_{j+1}, \ldots x_{j+n-1}$ \ respectively.
(That is, $t_1$ is a $k_1$-ary ordered term on \ 
$x_j, x_{j+1}, \ldots x_{j+k_1-1}$, \ $t_2$ is a
$k_2$-ary ordered term on the next $k_2$ variables,
and so on, where \ $n = k_1 + k_2 + \cdots k_m$.) \ 
Then \ $f(t_1, t_2, \ldots t_m)$ \ is an $n$-ary
ordered term on the variables \ 
$x_j, x_{j+1}, \ldots x_{j+n-1}$. \ We used 
ordered terms in groupoids in our earlier papers
\cite{BraittSilberger} and \cite{BraittHobbySilberger} 
and called them {\em formal products}.

Focusing on groupoids with operation $\star$, we see
that there are exactly $5$ different ordered terms on the
$4$ variables \ $x_1, x_2, x_3,x_4$.  They are:
\ $((x_1 \star x_2) \star x_3) \star x_4$, \ 
\ $(x_1 \star (x_2 \star x_3)) \star x_4$, \ 
\ $(x_1 \star x_2) \star (x_3 \star x_4)$, \ 
\ $x_1 \star ((x_2 \star x_3) \star x_4)$ \ and 
\ $x_1 \star (x_2 \star (x_3 \star x_4))$. \ 
As is well known (see \cite{Stanley}), a groupoid has \ 
$C(k-1) = (2k-2)! / {k! (k-1)!}$ \ many distinct
ordered terms on $k$ variables, where \ 
$C(m)$ \ is the $m$-th Catalan number.

Assume \ $k \geq 3$. \ Let \ $s(x_1, \ldots x_k)$ \ 
and \ $t(x_1, \ldots x_k)$ \ be distinct terms
of some groupoid \ $\langle G , \star \rangle$. \ 
If  \ $s(x_1, \ldots x_k) \neq t(x_1, \ldots x_k)$ \ 
for all \ $x_1, x_2, \ldots x_k \in G$, then we say
that $G$ {\em separates $s$ and $t$}.  The groupoid \ 
$\langle G , \star \rangle$ \ is 
{\em $k$-antiassociative} iff it separates all the
distinct pairs of ordered terms on \ 
$x_1, x_2, \ldots x_k$. \ 

Two observations are in order.  If $G$ is a groupoid
that separates two terms $s$ and $t$, then every 
subgroupoid of $G$ also separates $s$ and $t$.  Second,
suppose $G$ is a groupoid that separates $s$ and $t$,
and let $H$ be an arbitrary groupoid (with the same
operation symbol).  Then the Cartesian product \ 
$G \times H$ \ separates $s$ and $t$.

There are infinite groupoids that are
$k$-antiassociative for all $k$.  
One example is \ 
$\langle F^\sigma;\odot\rangle$, \ the set of all
formal products under a natural operation which is 
similar to concatenation. (See \cite{BraittSilberger} for a 
definition and proof.)  The free groupoid
(see \cite{BurrisSankappanavar}) on one or more 
generators is another example, as can be shown by a modification
of the proof of Theorem \ref{antiassociative theorem}.  
(At the end of the proof, where Theorem \ref{cover theorem}
is invoked, one argues directly instead.)

There are no finite groupoids
which are $k$-antiassociative for all $k$, since 
the number of $k$-ary ordered terms increases without
bound.  Once there are more terms than elements in
the groupoid, the Pigeonhole Principle implies that
there are terms which will not be separated in the groupoid.
This brings us to the following question, which we 
posed in \cite{BraittHobbySilberger}.

\begin{question}\label{question1}
\ For all \ $k \geq 3$, \ is there a finite
groupoid that is $k$--antiassociative?
\end{question}

By our observation above, this question may
be reduced to the following one.

\begin{question}\label{question2}
\ For each \ $k \geq 3$ \ and for all distinct ordered
terms $s$ and $t$ on \ $x_1, x_2, \ldots x_k$, \ is 
there a finite groupoid that separates $s$ and $t$?
\end{question}

An affirmative answer to the second question gives
an affirmative answer to the first.  To see this, 
assume that for all distinct ordered terms $s$ and $t$ 
on \ $x_1, x_2, \ldots x_k$, \ there is a finite groupoid 
\ $G_{s,t}$ \ that separates $s$ and $t$.  Then the product
of these groupoids separates all the $k$-ary ordered
terms, and is $k$-antiassociative.
The other direction is immediate, so the two questions
are equivalent.

Note also that whenever $3 \leq j < k$, a groupoid $\langle G, \star \rangle$
that is $k$-antiassociative is also $j$-antiassociative.  For suppose $s(x_1,x_2,\dots x_j)$
and $t(x_1,x_2,\dots x_j)$ are $j$-ary ordered terms that are not 
separated in $\langle G, \star \rangle$.
We let $r(x_{j+1},\dots x_k)$ be some fixed $(k-j)$-ary ordered term, and form 
$s'(x_1,x_2,\dots x_k) = s(x_1,x_2,\dots x_j) \star r(x_{j+1},\dots x_k)$ and
$t'(x_1,x_2,\dots x_k) = t(x_1,x_2,\dots x_j) \star r(x_{j+1},\dots x_k)$.  These are two
$k$-ary ordered terms that are not separated in $\langle G, \star \rangle$, a contradiction.

\S2 will present two preliminary examples.
We will answer Question \ref{question2} in the affirmative in
\S3, and generalize it to arbitrary groupoid terms
in \S4.  
\end{section}

\begin{section}{Preliminary examples}

We start with two simple constructions that often
yield groupoids separating two distinct $k$-ary
ordered terms.  The first is to simply take products
of deranged operations.  For example, define the 
operation $\Ltwo$ on the universe of $Z_2$ by setting
$x \Ltwo y = (x + 1) \mod 2$.  Then we have 
$(x \Ltwo y) \Ltwo z = (x+2) \mod 2$,  
$((x \Ltwo y) \Ltwo z) \Ltwo w = (x+3) \mod 2$, and so on.
The value of a term with leftmost variable $x$ is 
$(x+n) \mod 2$, where $n$ is the depth of $x$ in the term.
We also define $\Rthree$ on the universe of $Z_3$ by setting
$x \Rthree y = (y + 1) \mod 3$.  Similarly, we have that
the value of a term with rightmost variable $z$ is 
$(z+n) \mod 3$, where $n$ is the depth of $z$ in the term.

We consider the five possible $4$-ary 
ordered terms, which we list as follows: \ 

\noindent $t_1 = ((x_1 \star x_2) \star x_3) \star x_4$, \ 

\noindent $t_2 = (x_1 \star (x_2 \star x_3)) \star x_4$, \ 

\noindent $t_3 = (x_1 \star x_2) \star (x_3 \star x_4)$, \ 

\noindent $t_4 = x_1 \star ((x_2 \star x_3) \star x_4)$ \ and \ 

\noindent $t_5 = x_1 \star (x_2 \star (x_3 \star x_4))$. \ 

In $\langle Z_2, \Ltwo \rangle$, we have 
$t_1(w,x,y,z) = (w_1 + 3) \mod 2$, \ 
$t_2(w,x,y,z) =  (w_1 + 2) \mod 2$, \
$t_3(w,x,y,z) =  (w_1 + 2) \mod 2$, \
$t_4(w,x,y,z) =  (w_1 + 1) \mod 2$ and \
$t_5(w,x,y,z) =  (w_1 + 1) \mod 2$, \
so all the terms in $\{ t_1, t_4, t_5 \}$ are separated
from those in $\{ t_2 ,t_3 \}$ in this groupoid.  Similarly, 
the terms in the sets $\{ t_1 ,t_2 \}$, $\{ t_3 ,t_4 \}$
and $\{ t_5 \}$ are all separated from those in the
other sets in the groupoid $\langle Z_3, \Rthree \rangle$.
Continuing, all five terms are separated from each other in 
the product of the two groupoids.

The problem with this approach is that the value
of a term only depends on the depths of its leftmost and rightmost 
variables, so terms that have those two variables at the same depth can not be separated this way.

The next construction partially avoids this problem.
Suppose that \ 
$A = \langle A, + \rangle$ \ is an abelian group,
that $\alpha$ and $\beta$ are endomorphisms 
of \ $\langle A, + \rangle$, \ and that $c$ is a
fixed element of $A$.  We define an operation 
$\star$ on $A$ by setting \ 
$x \star y = \alpha(x) + \beta(y) + c$, \ and call
the groupoid \ $\langle A, \star \rangle$ \ the
{\em affine endomorphism groupoid} for $A$, $\alpha$, 
$\beta$ and $c$.  We denote this groupoid by \ 
$E(A,\alpha,\beta,c)$. \ 

As an example, suppose we want an affine endomorphism groupoid 
that separates the terms 
$s(v,w,x,y,z) = ((v \star w) \star (x \star y)) \star z$ and 
$t(v,w,x,y,z) = ((v \star (w \star x)) \star y) \star z$.
In both terms, $v$ has depth 3 and $z$ has depth 1, so the previous 
approach can't succeed.

In \ $E(A,\alpha,\beta,c)$, \ we get
$s(v,w,x,y,z) = 
((\alpha(v)+\beta(w)+c) \star (\alpha(x)+\beta(y)+c)) \star z
= (\alpha^2(v)+\alpha \beta(w)+\alpha(c)+\beta \alpha(x) +
\beta^2(y)+\beta(c)+c) \star z
= \alpha^3(v)+\alpha^2 \beta(w)+\alpha^2(c)+
\alpha \beta \alpha(x) + \alpha \beta^2(y)+
\alpha \beta(c)+\alpha(c) + \beta(z) + c$.
This is quite messy, so we make the simplifying assumptions
that \ $\alpha^3 = \alpha^2$, that \ $\beta^2 = \beta$, \ and that
$\alpha$ and $\beta$ commute.  This gives us $s(v,w,x,y,z) = 
\alpha^2(v)+\alpha^2 \beta(w)+
\alpha^2 \beta(x) + \alpha \beta(y)+ \beta(z)+
\alpha^2(c)+\alpha \beta(c)+\alpha(c)  + c$. \ 
And a similar calculation gives
$t(v,w,x,y,z) = 
\alpha^2(v)+\alpha^2 \beta(w)+
\alpha^2 \beta(x) + \alpha \beta(y)+ \beta(z)+
\alpha^2\beta(c)+\alpha^2(c)+\alpha(c)+c$. \

Observe that both terms have the identical portion \ 
$\alpha^2(v)+\alpha^2 \beta(w)+
\alpha^2 \beta(x) + \alpha \beta(y)+ \beta(z)$, 
and only differ in their constants.  (Our choice of simplifying 
assumptions was designed to do this.)  So we can separate
the terms by insuring that 
$\alpha^2(c)+\alpha \beta(c)+\alpha(c)+c$ and
$\alpha^2\beta(c)+\alpha^2(c)+\alpha(c)+c$ 
have different values.

Fortunately, there are $A$, $\alpha$, $\beta$ and $c$ 
that satisfy these conditions.  We may work over \ ${\mathbf Z}_2$, \ and consider $2 \times 3$
matrices with elements in \ ${\mathbf Z}_2$. \ This gives us that the 
group $A$ is isomorphic to \ ${\mathbf Z}_2^6$, \ a $64$-element group.
The desired actions of $\alpha$ and $\beta$ on $A$ can be 
realized by letting $\beta$ copy the top row of $A$ onto
the bottom row, and by letting $\alpha$ copy the left column 
of $A$ onto the middle column and the middle column onto the right column.

$$
\mbox{That is,   } \alpha \left[
\begin{array}{c c c}
	d & e & f \\
	g & h & i 
\end{array}
\right] = \left[
\begin{array}{c c c}
	d & d & e \\
	g & g & h 
\end{array}
\right]   
\mbox{ and  }  \beta \left[
\begin{array}{c c c}
	d & e & f \\
	g & h & i 
\end{array}
\right] = \left[
\begin{array}{c c c}
	d & e & f \\
	d & e & f 
\end{array}
\right] \mbox{.} 
$$
$$
\mbox{Finally, we take $c =$  } \left[
\begin{array}{c c c}
	1 & 0 & 0 \\
	0 & 0 & 0
\end{array}
\right] 
\mbox{. \ This gives \ } \alpha \beta (c) =
\left[
\begin{array}{c c c}
	1 & 1 & 0 \\
	1 & 1 & 0
\end{array}
\right] 
\mbox{ \ and}
$$
$$
\alpha^2 \beta (c) =
\left[
\begin{array}{c c c}
	1 & 1 & 1 \\
	1 & 1 & 1
\end{array}
\right] 
\mbox{, so \ }
s(\vec{0}) = 
\left[
\begin{array}{c c c}
	0 & 1 & 1 \\
	1 & 1 & 0 
\end{array}
\right] 
\mbox{and $t(\vec{0}) =$} \left[
\begin{array}{c c c}
	0 & 1 & 0 \\
	1 & 1 & 1
\end{array}
\right]   
\mbox{.}
$$

The above technique requires making assumptions
about $\alpha$ and $\beta$ in order to simplify the
expressions for the terms.  One has some latitude with the
assumptions.  For example, one may take \ 
$\alpha^{k+1} = \alpha^k$, or $\beta^{k+1} = \beta^k$
for any value of $k$, and no longer require that
$\alpha$ and $\beta$ commute.  But a point is reached
where that no longer helps.  We were unable to use the
above method to produce a groupoid that separated the
two $5$-ary terms \ 
$s = (x_1 \star (x_2 \star x_3)) \star (x_4 \star x_5)$ \ 
and \ 
$t = (x_1 \star x_2) \star ((x_3 \star x_4) \star x_5)$. \ 
(These terms are represented by trees in Figure \ref{fig:trees}.)

So we turn to another method, which we will present
in the next section.
\end{section}

\begin{section}{Finite $k$-antiassociative groupoids}
We will use a somewhat involved construction, and will require some
preliminary definitions.  Recall that a {\em full binary tree}
is a rooted tree where every internal node has exactly two 
children.  (For further definitions and theorems, see \cite{Knuth}
or a recent text in discrete mathematics or data structures.)

When full binary trees are used as data structures, the
two nodes directly below each internal node are called 
its {\em left} and {\em right} children, and the subtrees
with these children as roots are the {\em left} and {\em right subtrees}
of that node.  As is well known, groupoid terms correspond
to full binary trees with leaves labeled by variables.  
If $s$ is a groupoid
term, we will denote the corresponding tree by \ $T(s)$. \ 
This correspondence may be defined recursively as follows.  
If $s$ is a single variable $x_i$, then \ $T(s)$ \ is a
tree with one node, labelled $x_i$.  If $s$ and $t$ are 
groupoid terms, then \ $T(s \star t)$ \ is the tree with
a root that has \ $T(s)$ \ as its left subtree and 
\ $T(t)$ \ as its right subtree.

We will also label the nodes of binary trees with strings
made from the characters `$l$' and `$r$'.  As is usual, 
we will write the set of all such strings as \ $\{l,r\}^*$. \ 
In dealing with 
strings, we will show concatenation by simply writing 
the two strings next to each other.  We use $\Lambda$ to
denote the empty string, which is the identity for concatenation.  
Our labeling may 
be defined recursively as follows.  

The root is labeled $\Lambda$.
If a node is labeled $a$, then its left and right 
children are labeled \ $al$ \ and \ $ar$, \ respectively.
These labels may be thought of as directions for how to 
get to a node by starting at the root and turning the 
correct way at each branching.  

Given a string $p$, an {\em initial substring} of $p$ is
a string $q$ so that \ $p = qu$ \ for some string $u$.
(Note that the empty string $\Lambda$ is an initial substring
of every string.)
A substring is {\em proper} if it is not equal to the
entire original string, and {\em nontrivial} if it is
not equal to $\Lambda$.

Putting these two ideas together, occurrences of variables in 
a groupoid term $s$ correspond to leaves of \ $T(s)$. \ 
The string that is the label of the leaf corresponding to an 
occurrence of the variable $x_i$ will be called the {\em path} 
of that occurrence.  If $x_i$ only occurs once, we may also call 
this the path of $x_i$.
Generalizing this, for any subterm $b$ of $s$, we have that the 
{\em path} of $b$ is also the label of the interior node of \ 
$T(s)$ \ corresponding to the root of subtree $T(b)$.
  
\begin{figure}[h]
	\centering
		\includegraphics{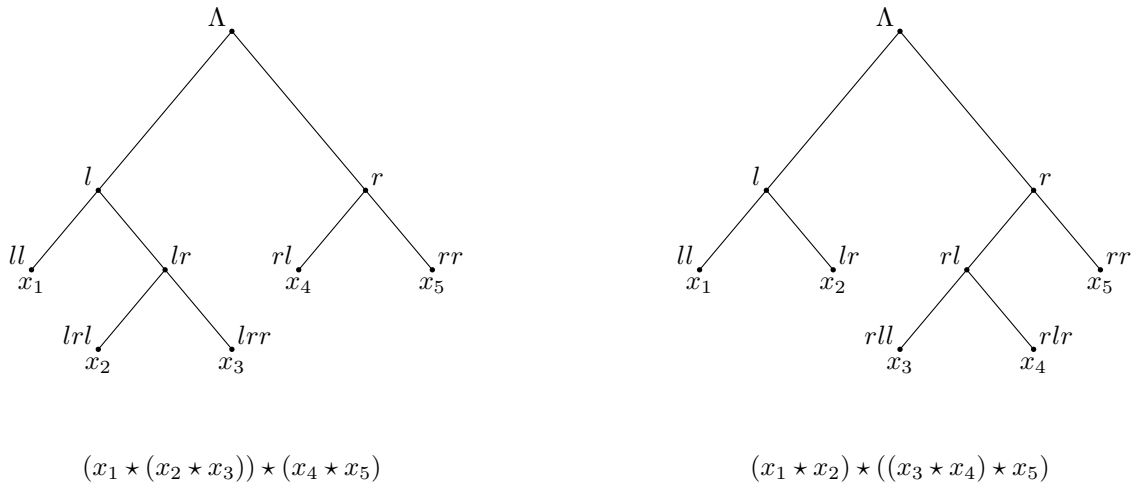}
	\caption{Trees for two terms}
	\label{fig:trees}
\end{figure}

For example, consider \ 
$s = (x_1 \star (x_2 \star x_3)) \star (x_4 \star x_5)$. \ 
We have \ $\path(x_1) = ll$, \ $\path(x_2) = lrl$, \ 
$\path(x_3) = lrr$, \ $\path(x_4) = rl$, \ $\path(x_5) = rr$ \ 
and \ $\path(x_2 \star x_3) = lr$. \ 
(When there is danger of confusion,
we will write \ $\path_s(x_i)$ \ to show we mean the
path in the term $s$.)  The tree for this term is on the
left side of Figure \ref{fig:trees}.

If $s$ is a groupoid term, we use \ $\Paths(s)$ \ for the
set of all paths to variables in $s$.  Similarly, we have
paths to the internal nodes of the tree \ $T(s)$; \ these
correspond to proper initial substrings of paths to the
leaves of \ $T(s)$. \ Given the term $s$ with $q$ the path 
to a node of $T(s)$, we let $s_q$ denote the subterm of $s$ 
with $T(s_q)$ rooted at the node $q$ of $T(s)$.
Then for any subterm $b$ of $s$, if we let $q$ be the path
of $b$ in $s$, we have $b = s_q$.

Our long-term goal is to form a groupoid that separates
any two distinct $k$-ary ordered terms $s$ and $t$.  
We will need some preliminary ideas in order to do this.
Our groupoids will have elements which are vectors
of finite length over the $2$-element field ${\mathbf Z}_2$.
We take the index set of the components of these vectors
to be the set of natural numbers
 \ ${\mathbf N} = \{0,1,2, \dots \}$. \  All of our vectors
 will only have a finite number of components, or equivalently,
 will be vectors indexed by ${\mathbf N}$ that are zero in all 
 but finitely many components.  Given any finite set of 
 such vectors, we let $M$ be the set of all indices where any
 of them have nonzero components.  Then all these vectors lie in the finite
 subspace consisting of vectors with all their
 components outside 
 of $M$ equal to $0$.  We will usually leave this final 
 reduction to a finite groupoid to the reader.
 
 We will actually be using only the additive structure of
 the field 
 ${\mathbf Z}_2$
 , and viewing it as an abelian 
 group.  Our groupoids will all be affine endomorphism 
 groupoids, although the endomorphisms will be built up from 
 their actions on the components of vectors.  One nice 
 consequence of this is that we will be able to add groupoid
 operations pointwise.  If $\star_1$ and $\star_2$ are two
 groupoid operations on vectors over ${\mathbf Z}_2$, their
 sum \ $\star_1 + \star_2$ \ will be defined by \ 
$\vec{x} \,(\star_1 + \star_2)\, \vec{y} = (\vec{x} \star_1 \vec{y}) + 
(\vec{x} \star_2 \vec{y})$. \ Since we are working over
${\mathbf Z}_2$, all additions of values such as the above are 
done modulo $2$.  We will periodically note this fact, but not always.

We will define groupoid operations by their actions on components.
In this section we will use the convention that the vectors $x$, $y$ and $z$ are such 
that \ $z = x \star y$ \ for our groupoid operation $\star$.  We
will also simply write $x$ instead of $\vec{x}$, and write
$x[a]$ for the $a$-th component of the vector $x$. (For clarity, we 
will always use square brackets for this.)  To specify a
groupoid operation, it then suffices to say what $z[i]$ is for all
$i$.  We will do this by giving a sequence of equations for the
$z[i]$.  To emphasize that values are being assigned to the $z[i]$,
we will use $:=$ instead of the normal equality symbol.  One further
convention is that each $z[i]$ will be zero, unless that $z[i]$ is 
explicitly assigned a value.

For example, consider the groupoid operation which we will later 
call \ $\left\|2,lr,0\right\|$. \  We define it by the two equations \ 
$z[0] := x[a]$ \ and $z[a] := y[2]$. \  The only indices mentioned
are $0$, $a$ and $2$, so we can focus on just those three components,
and view our vectors as $3$-tuples.  Writing our operation as $\star$,
we have \ $\langle x[0],x[a],x[2] \rangle \star 
\langle y[0],y[a],y[2] \rangle = \langle x[a],y[2],0 \rangle$. \ 
Continuing to use $\star$ for this operation, consider the term \ 
$s = (u \star v) \star w$.  We have \ 
$u \star v = \langle u[a],v[2],0 \rangle$, \ and \ 
$(u \star v) \star w = \langle u[a],v[2],0 \rangle \star 
\langle w[0],w[a],w[2] \rangle = \langle v[2],w[2],0 \rangle$. \ 
The $0$-th component of $s$ is the $2$nd
component of $v$, where $\path_s(v) = lr$.  \ This motivates calling
the operation \ $\left\|2,lr,0\right\|$. \ 

When using the operation $\left\|2,lr,0\right\|$, we will 
be looking only at the $0$-th component of the output,
and ignoring the $a$-th component.  With this 
understanding, it makes little difference what the
index $a$ is.
  So we will assume that indices
such as $a$, $b$ and so on are always chosen to minimize 
{\em collisions}.  This means that no indices will be equal unless
they are explicitly represented with equivalent expressions.
This can be easily achieved by appropriate choices of values for
$a$, $b$ and so on, and will not jeopardize the finiteness of
any groupoids we produce.  As long as there are no collisions,
groupoids obtained for different values of $a$ will be isomorphic.
Accordingly, we will speak of {\em the} groupoid operation \ 
$\left\|2,lr,0\right\|$, \ and so on.

\begin{df}\label{basic operation def}
Let \ $p = p_0p_1p_2 \cdots p_j$ \ be a nonempty string 
in \ $\{l,r\}^*$, \ 
and let $m$ and $n$ be natural numbers.  Then the 
operation \ $\left\|m,p,n\right\|$ \ is defined via the following
equations,
where we assume that $a,a+1,\dots a+j+1$ are distinct
from $m$ and $n$.  If $p_0$ is $l$, the first equation
is  \ $z[n] := x[a]$, and if $p_0$ is $r$, it is
\ $z[n] := y[a]$.  \ If \ $p_1 = l$, \ the next equation
is  \ $z[a] := x[a+1]$, and if \ $p_1 = r$, \ it is
\ $z[a] := y[a+1]$.  This pattern continues, with
\ $z[a+i] := x[a+i+1]$ \ if \ $p_{i+1} = l$ \ or \ 
$z[a+i] := y[a+i+1]$ \ if \ $p_{i+1} = r$, \ for 
all \ $i \leq j-2$. \ The last equation is \ 
$z[a+j-1] := x[m]$ \ if \ $p_{j} = l$ \ and it is \ 
$z[a+j-1] := y[m]$ \ if \ $p_{j} = r$. \ 
\end{df}

The idea is that \ $\left\|m,p,n\right\|$ \ transfers the value
of the $m$-th component of the vector with path $p$ in 
the term $s$ to the $n$-the component of the result 
of $s$, with as few side effects as possible.  We
are assuming that none of the indices used to define
\ $\left\|m,p,n\right\|$ \ is equal to any of the others, except that
possibly \ $m = n$. \  In 
other words, the operation \ $\left\|m,p,n\right\|$ \ is 
{\em duplicate free}.  If $m_1$ is distinct from both
 $m_2$ and $m_0$, and $p$ and $q$ are strings in \ 
$\{l,r\}^*$, \ then the operation \ 
$\left\|m_2,q,m_1\right\| + \left\|m_1,p,m_0\right\|$ \ is duplicate free 
by our convention that indices are chosen to minimize
collisions.  In isolation, the sum \ 
$\left\|m_2,q,m_1\right\| + \left\|m_1,p,m_0\right\|$ \ is equivalent to \ 
$\left\|m_2,pq,m_0\right\|$. \ The one difference is that the
former explicitly mentions the index $m_1$.  We will henceforth assume
that all our groupoid operations are duplicate free.

\begin{lemma}\label{path lemma}
Let $\star$ be a duplicate and collision free groupoid operation that
contains \ $\left\|m,p,n\right\|$ \ as a summand, and let $s$ be a 
groupoid term where $p$ is the
path to a node of $T(s)$.  Letting $s_p$ be the subterm
of $s$ at that node, \ $s[n] = s_p[m]$ for all values of the variables of $s$. \ 
\end{lemma}

\begin{proof}
Since $\star$ is duplicate and collision free, the only summand of $\star$
that affects the value of $s[n]$ is $\left\|m,p,n\right\|$.  So we may ignore
the rest of $\star$, and assume $\star$ is $\left\|m,p,n\right\|$.
Letting \ $p = p_0p_1p_2 \cdots p_j$, \ 
we will prove the lemma by induction on $j$.   
Our basis is when $j = 0$, making the operation \ 
$\left\|m,p_0,n\right\|$. \  We will do the case where \ $p_0 = l$, \ 
the one for \ $p_0 = r$ \ is similar.  Now \ 
$s = s_l \star s_r$, \ where $\star$ is \ $\left\|m,l,n\right\|$. \ 
The one relevant assignment is \ $z[n] := x[m]$, \ giving
\ $s[n] = z[n] = x[m] = s_l[m]$, \ as desired.

For the induction step, assume the statement is true
for \ $j-1$, \ and that we want to show it for the
path \ $p = p_0p_1p_2 \cdots p_j$. \  We write \ 
$\star = \left\|m,p,n\right\|$ \ as \ 
$\left\|m,p_j,b\right\| + \left\|b,p_0p_1p_2 \cdots p_{j-1},n\right\|$ \ 
for some new index $b$, and let $q$ be \ 
$p_0p_1 \cdots p_{j-1}$, \ so \ $p = qp_j$. \ 
By the statement for \ $j-1$, \ $s[n] = s_q[b]$.  \   We have \ 
$s_q[b] = (s_{ql} \star s_{qr})[b] = (s_{ql} \,\, \left\|m,p_j,b\right\| \,\, s_{qr})[b]$, \ 
where the last step follows because indices are chosen to 
minimize collisions.  There are now two cases.  
We will do the one for \ $p_j = r$; \ 
the case for \ $p_j = l$ \ is similar.
Since \ $p_j = r$, \ we have \ $z[b] := y[m]$ \ in
\ $\left\|m,p_j,b\right\|$. \ So \ $s_q[b] = s_{qr}[m] = s_p[m]$, \ 
since \ $qr = qp_j = p$. \ Thus \ 
$s[n] = s_q[b] = s_p[m]$, as desired.
\end{proof}

Given the groupoid operation \ $\left\|m,p,n\right\|$, \ we define
the {\em tweaked} operation \ $\left\|m,p,n\right\|'$ \ to be identical
to \ $\left\|m,p,n\right\|$ \ except for one assignment.  Writing $p$
as \ $p_0q$, \ $\left\|m,p,n\right\|$ \ has an assignment of the form \ 
$z[n] := x[k]$ \ if \ $p_{0} = l$ \ and one of the form \ 
$z[n] := y[k]$ \ if \ $p_{0} = r$. \ Whichever one occurs,
we modify it by adding $1$, giving \ 
$z[n] := (x[k] + 1) \mod 2$ \ if \ $p_{0} = l$ \ or giving  \ 
$z[n] := (y[k] + 1) \mod 2$ \ if \ $p_{0} = r$. \ 

A slight modification of the proof of the previous lemma
then establishes the following.

\begin{lemma}\label{tweaked lemma}
Let $\star$ be a duplicate and collision free groupoid operation that
contains \ $\left\|m,p,n\right\|'$ \ as a summand, and let $s$ be a 
groupoid term where $p$ is the
path to a node of $T(s)$.  Letting $t = s_p$ be the subterm
of $s$ at that node, \ $s[n] = (t[m] + 1) \mod 2$. \ 
\end{lemma}

We are now ready to establish a powerful theorem, 
which holds for all groupoid terms regardless of any
conditions on the order or number of appearances of 
variables.

\begin{theorem}\label{cover theorem}
Let $s$ and $t$ be any groupoid terms.  Suppose 
the variable $x$ has an occurrence in $s$ where the
path to that occurrence is $p$, and that $x$ has an
occurrence in $t$ where the path to that occurrence
is $q$.  Then if $q$ is a proper initial substring
of $p$, the terms $s$ and $t$ can be separated.
\end{theorem}

\begin{proof}
Let $s$, $t$, $x$, $p$ and $q$ be as above.  By 
hypotheses, \ $p = qw$ \ for a nonempty string $w$.
We let $\star$ be \ $\left\|1,q,0\right\| + \left\|1,w,1\right\|'$. \ 

First consider the value of $t[0]$ for this $\star$.
Since \ $\left\|1,w,1\right\|'$ \ 
does not have an assignment to $z[0]$, \ $\left\|1,w,1\right\|'$
makes \ $t[0] = 0$, \ and we can ignore it.
As for  \ $\left\|1,q,0\right\|$, \ Lemma \ref{path lemma} gives \ 
$t[0] = t_q[1] = x[1]$. \  This implies that $\star$ 
sets $t[0] = x[1].$

Now consider the value of $s[0]$ for the above $\star$.
As in our calculation for $t[0]$, we have \ $s[0] = s_q[1]$. \ 
But now $s_q$
is a nontrivial subterm of $s$, so we compute \ $s_q[1]$. \ 
The operation \ $\left\|1,q,0\right\|$ has no effect on \ $s_q[1]$, \ 
so we ignore it and just consider the effect of \ 
$\left\|1,w,1\right\|'$. \  It gives \ $s_q[1] = s_{qw}[1] + 1$, \ by Lemma
\ref{tweaked lemma}.  Putting these together, we have \ 
$s[0] = s_q[1] = s_{qw}[1] + 1 = s_p[1] + 1 = x[1] + 1$. \ 
This shows that $s$ and $t$ always have different values
in a finite groupoid, since it is always true
that \ $s[0] \neq t[0]$.
\end{proof}  
 
\begin{theorem}\label{antiassociative theorem}
\ For all \ $k \geq 3$ \ there is a $k$-antiassociative finite groupoid.
\end{theorem}

\begin{proof}
It is enough to produce a finite groupoid that 
separates any two distinct $k$-ary ordered terms 
$s$ and $t$.  Given any two distinct terms $s$ and $t$ 
with $k \geq 3$, we let $x_m$ be the leftmost variable on which 
$s$ and $t$ do not {\em agree}, in the sense that \ 
$\path_s(x_i) = \path_t(x_i)$ \ for all \ 
$i < m$, \ and \ $\path_s(x_m) \neq \path_t(x_m)$. \ 

We claim that for any two such distinct $k$-ary 
terms $s$ and $t$, one of \ $\path_s(x_m)$ \ 
or \ $\path_t(x_m)$ \ is a proper initial substring of 
the other.  The proof is by induction on $j$, where
$j$ is the minimum of the lengths of $\path_s(x_m)$ and
$\path_t(x_m)$.  If $j = 0$, then either $s = x_m$ or
$t = x_m$.  
Without loss of generality, assume $s = x_m$.
Then $\path_s(x_m) = \Lambda$.  If $\path_t(x_m)$ is also
$\Lambda$, we have $s = x_m = t$, a contradiction.  So
$\path_t(x_m) \neq \Lambda$, and $\path_t(x_m)$ has 
$\path_s(x_m)$ as a proper initial substring.  This establishes 
the basis case.

For the induction step, suppose that the claim is true
for $j-1$, and that we want to prove it for $j$.  We have
that \ $s = s_l \star s_r$, \ and \ $t = t_l \star t_r$.  
We have two cases, depending on where $x_m$ occurs.

If $x_m$ occurs in $s_r$, then $x_m$ also occurs in $t_r$
since \ $s_l = t_l$ \ because $s$ and $t$ agree for all \ 
$i < m$. \ But then $x_m$ is the leftmost variable on 
which $s_r$ and $t_r$ disagree, so one of \ $\path_{s_r}(x_m)$ \ 
and \ $\path_{t_r}(x_m)$ \ is a proper initial substring of 
the other, by the statement for $j-1$.  
Since \ $\path_s(x_m)$ \ 
and \ $\path_t(x_m)$ \ are obtained from these paths by adding
$r$ to the start, one of them is also a proper initial 
substring of the other.

So suppose $x_m$ occurs in $s_l$.  As in the previous paragraph,
if $x_m$ occurred in $t_r$, we would get that $x_m$ occurred in $s_r$.
Thus $x_m$ occurs in $t_l$.  Then $x_m$ is the leftmost variable on 
which $s_l$ and $t_l$ disagree, and one of \ $\path_{s_l}(x_m)$ \ 
and \ $\path_{t_l}(x_m)$ \ is a proper initial substring of 
the other.  Adding $l$ to the start of these paths gives
\ $\path_s(x_m)$ \ and \ $\path_t(x_m)$, \ so one of them is 
a proper initial substring of the other.  This proves the claim.

Now let distinct $k$-ary $s$ and $t$ with $k \geq 3$ be given.
The claim gives us that one of \ 
$\path_t(x_m)$ \ and $\path_s(x_m)$ is a proper initial 
substring of the other.\
\ We apply Theorem \ref{cover theorem}, and
obtain a finite groupoid that separates $s$ and $t$.
\end{proof}
\end{section}

\begin{section}{Separating arbitrary groupoid terms}

We can generalize
the questions of the previous section, by relaxing 
the condition that each variable appears
once in every term in order of their indices.  

As before, we can reduce everything to the problem of
finding finite algebras that separate pairs of terms.
(Theorem \ref{TFAE theorem} uses free algebras to 
give us a condition for when infinite algebras 
exist that separate a pair of terms.)

If we try to separate the two groupoid terms \ 
$s(x,y) = x \star y$ \ and \ $t(x,y) = y \star x$, \ 
we rapidly run into trouble.  When $x = y$, both
terms reduce to \ $x \star x$, \ so it is impossible
to separate them in any groupoid.  This trick of 
identifying variables can be applied whenever $s$ and
$t$ have the same {\em shape}, which we can define
rigorously as follows.  Let $\chi$ be a distinguished
variable symbol, that we agree to use nowhere else.
Then we simply define the {\em shape} of a term \ 
$s(x_1,x_2, \ldots x_k)$ \ to be the term \ 
$s(\chi,\chi, \ldots \chi)$. \ 

As an aside, note that we can easily make 
the term functions \ $x \star y$ \ and \ $y \star x$ 
not equal whenever $x \neq y$,
for instance by letting $\star$ be $-$ over
$Z_3$.  This prompts the following question, which
we will not deal with further in this paper.

\begin{question}\label{same shape question}
\ Suppose that $s$ and $t$ are two terms of the 
same shape, and let $x_1$, $x_2$, \ldots $x_k$ be
all the variables appearing in either of them.
Given a set $S$ of equalities between variables in $\{ x_1,x_2, \ldots x_k \}$,
there is a function $\phi$ from $\{ 1,2,\dots k\}$ into $\{ 1,2,\dots k\}$
such that for each $i$, $\phi(i)$ is the least number such that the equality
$x_i = x_{\phi(i)}$ can be deduced from equalities in $S$.
Also let $s'$ be the term that results when for all $i$, $x_i$ is replaced by $x_{\phi(i)}$ throughout $s$, and let $t'$ be defined similarly.

Call a set of equalities $S$ between variables in $\{ x_1,x_2, \ldots x_k \}$
{\em identifying} iff the terms $s'$ and $t'$ are the same.  When is it
possible to have a finite algebra where the term functions
$s$ and $t$ are not equal whenever the values of their variables
do not satisfy any identifying sets of equations?
\end{question}

From now on, we will focus on separating two 
groupoid terms of different shapes.  
Since we are now dealing with arbitrary terms,
variables may occur more than once in a given 
term.  For clarity, we will usually use primes
to distinguish occurrences of a variable from the
variable itself, so that $x'$ might denote some 
particular occurrence of $x$.
We will say that terms $s$ and
$t$ are {\em finitely separated} whenever they are 
separated in some finite groupoid.

Observe that any groupoid term \ $s$ \ 
has a {\em natural order} to the occurrences of
its variables, the order produced
by an inorder transversal of the leaves of its full 
binary tree $T(s)$.  We will always write terms by 
listing occurences of variables
in this natural order.  In this
case, we call $x_1'$ the {\em leftmost} variable occurrence
in $s(x_1,...)$.  Each variable occurrence in $s$ 
corresponds to a leaf in $T(s)$, so occurrences of a given
variable may be distinguished by their paths in $T(s)$.
The leftmost variable occurrence in $s$ is then the only
one with a path in $\{l\}^*$.

By the {\em depth} of an occurrence of a
variable in the term $s$, we mean its height in $T(s)$.
We will denote the depth in $s$ of the variable occurrence $x'$
by $d_s(x')$.  Note that this is the same as the length of
the string $\path_s(x')$.

A naive intuition would be that terms $s$ and $t$ could
not be separated when there were a number of variables 
occurring in one term and not the other.  It is certainly
true that having more variables of this sort gives more
possibilities to assign values to them that would force
$s$ and $t$ to be equal.  For example, let
$s$ be \ $(x \star y) \star z$, \ and let $t$ be \ 
$(x \star x) \star (x \star x)$. \  Letting $x$ have
any fixed value, we assign \ $y := x$ \ and \ 
$z := x \star x$. \ Substituting these values in $s$,
it becomes \ $(x \star x) \star (x \star x)$, \ which
is $t$.  So $s$ and $t$ can not be separated in any
groupoid.  

However, there are terms with only a single variable
in common that can still be separated in a finite
groupoid.  For example, let $s$ be \ $x \star p$ \ and
let $t$ be \ $(x \star y) \star q$, where $p$ and $q$
can be arbitrary terms on any variables.  For the leftmost
occurrences of $x$, we have \ $\path_s(x) = l$ \ and \ 
$\path_t(x) = ll$.  So Theorem \ref{cover theorem} gives a finite
groupoid that separates $s$ and $t$.

To continue our investigation, we need the following 
extension of Theorem \ref{cover theorem}, which requires further 
definitions to state.  If $s$ and $t$ are groupoid terms
and $y$ and $z$ are variables, we say 
that $y$ {\em occurs above} $z$ if there are occurrences
$y'$ of $y$ and $z'$ of $z$ so that either 
$\path_s(y')$ is a initial substring of 
$\path_t(z')$ or $\path_t(y')$ is a initial 
substring of $\path_s(z')$.  In this situation, we also say that the
occurrence $y'$ is {\em above} the occurrence $z'$.
Similarly, $y$ {\em occurs strictly above} $z$ if there are occurrences
$y'$ of $y$ and $z'$ of $z$ so that either 
$\path_s(y')$ is a proper initial substring of 
$\path_t(z')$ or $\path_t(y')$ is a proper initial 
substring of $\path_s(z')$.  

We say that terms $s$ and
$t$ {\em have a cycle} if there is a sequence of variables
$y_0,y_1, \ldots y_{m-1}$ where $y_0$ occurs above $y_1$,
$y_1$ occurs above $y_2$, and so on, ending with 
$y_{m-1}$ occurring above $y_0$, where at least one of these
occurrences is strictly above the other.  
The hypothesis of Theorem
\ref{cover theorem} is that a single variable $x$ occurs
above itself, so that $s$ and $t$ have a cycle of length 1,
where the sequence $y_0,y_1, \ldots y_{m-1}$ is just $x$.
Our next theorem extends this result to cycles of arbitrary
length.  

\begin{figure}[ht]
	\centering
		\includegraphics{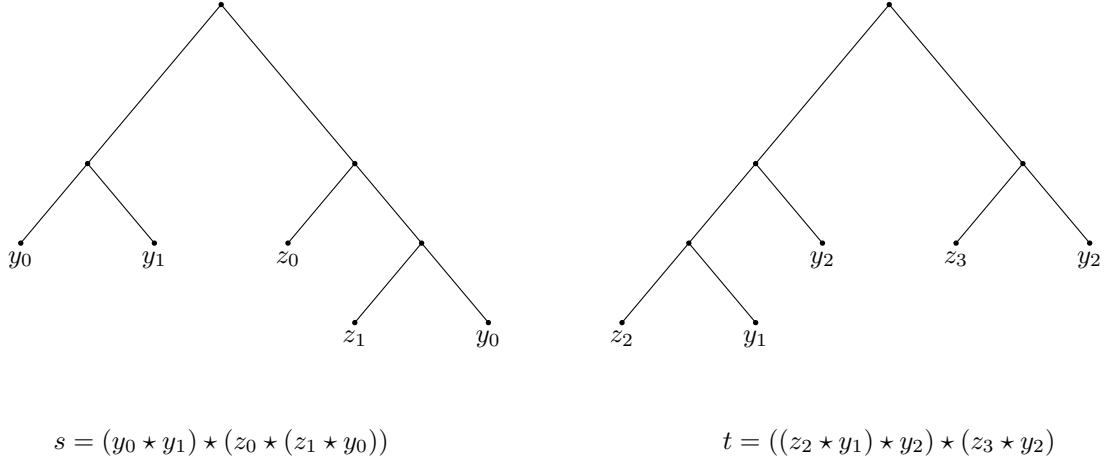}
	\caption{Two terms with a cycle}
	\label{fig:cycle}
\end{figure}

This proof will be easier to follow if we have an example for
reference.  It may be useful to refer back to this example while
reading the proof, as some of the notation it uses is defined
in the proof.  Figure \ref{fig:cycle} shows a cycle $y_0 y_1 y_2$
of length $3$, where $s = (y_0 \star y_1)\star(z_0 \star (z_1 \star y_0))$
and $t = ((z_2 \star y_1)\star y_2) \star (z_3 \star y_2))$.
Matching the notation of the coming theorem, we use superscripts
of $u$ and $d$ (for ``up''and ``down'') to label the distinct occurrences 
of variables in the cycle, as shown in Table \ref{table:upanddown}.

\begin{table}[ht]
	\centering
		\begin{tabular}{c|c|c|l|l|l}
		index&occurrence&term&path&$p_i$&$q_i$\\
		\hline
		$0$ & $y_0^u$ & $s$ & $ll$  & $ll$ &      \\
		    & $y_1^d$ & $t$ & $llr$ &      & $r$  \\
		\hline
		$1$ & $y_1^u$ & $s$ & $lr$ & $lr$ &           \\
		    & $y_2^d$ & $t$ & $lr$ &      & $\Lambda$ \\
		\hline
		$2$ & $y_2^u$ & $t$ & $rr$  & $rr$ &      \\
		    & $y_0^d$ & $s$ & $rrr$ &      & $r$  \\			
		\end{tabular}
	\caption{Occurrences in a cycle}
	\label{table:upanddown}
\end{table}

In the cycle, $y_0$ is strictly above $y_1$, since the occurrence 
$y_0^u$ in $s$ has path $ll$, which is an initial substring of $llr$,
the path in $t$ of the occurrence $y_1^d$.  And $y_1$ is above 
(but not strictly above) $y_2$, since the occurrence 
$y_1^u$ in $s$ has path $lr$, which is a (non-proper) initial 
substring of $lr$, the path in $t$ of the occurrence $y_2^d$.
Finally, $y_2$ is strictly above $y_0$, since the occurrence 
$y_2^u$ in $t$ has path $rr$, which is an initial substring of $rrr$,
the path in $s$ of the occurrence $y_0^d$.

The theorem also defines relations $\sim$ and $\approx$ on the index
set, which is $I = \{0,1,2\}$ in our example.  We have $1 \sim 2$, since 
$y_1^u$ is not strictly above $y_2^d$.  The relation $\approx$ is the
equivalence relation generated by $\sim$, so its classes are $\{0\}$
and $\{1,2\}$.  The function $f$ that takes each $i \in I$ to the least
element in its $\approx$ class has $f(0) = 0$ and $f(1) = f(2) = 1$.
Finally, the operation $\star'$ is 
$\left\| 3,ll,0 \right\| + \left\| 4,lr,1  \right\| + \left\| 4,rr,2 \right\| + 
\left\| 4,r,3 \right\|' + \left\| 3,r,4 \right\|$.  The reader can verify 
that this operation makes
$s[0] + s[1] + s[2] = y_0[3] + y_1[4] + (z_1 \star' y_0)[4] =
y_0[3] + y_1[4] + y_0[3] = y_1[4]$, where the last step follows since we are
adding values modulo $2$.  Similarly, 
$t[0] + t[1] + t[2] = (z_2 \star' y_1)[3] + y_2[4] + y_2[4] =
y_1[4] + 1 + y_2[4] + y_2[4] = y_1[4] + 1$, which always has a different value.

\begin{theorem}\label{extended theorem}
Let $s$ and $t$ be groupoid terms which have a 
cycle.  Then $s$ and $t$ are separated in a finite
groupoid.
\end{theorem}

\begin{proof}
Let $s$ and $t$ be terms with a cycle as above.  
So we have a sequence of variables
$y_0,y_1, \ldots y_{m-1}$ where $y_0$ occurs above $y_1$,
$y_1$ occurs above $y_2$, and so on, ending with 
$y_{m-1}$ occurring above $y_0$.  We may assume that
this cycle has minimal length $k$ for all cycles of $s$ 
and $t$, and that $k \geq 2$ since cycles of length $1$ 
are covered by Theorem \ref{cover theorem}.  This
implies that all of the variables $y_i$ are
distinct.  We also adopt the convention that our 
subscripts are calculated modulo $k$, so that $y_k$ is 
the same as $y_0$.

Each of the $y_i$ has two occurrences in the cycle.
For each $i$, let $y_i^u$ be the occurrence of 
$y_i$ that is above an occurrence of $y_{i+1}$, and
let $y_i^d$ be the occurrence of $y_i$ that is below
an occurrence of $y_{i-1}$.
A given occurrence $y'$ of a variable may be either in the 
term $s$ or in the term $t$.  

We denote whichever of $s$ and $t$ an occurrence
$y'$ is in by $\term(y')$.  We will 
then write $\path(y')$ to denote the path of $y'$ in 
$\term(y')$.  Note that $\term(y_i^u) \neq \term(y_{i+1}^d)$
for all $i$, since $\path(y_i^u)$ is an initial substring of
$\path(y_{i+1}^d)$ and $y_i \neq y_{i+1}$.

We will denote $\path(y_i^u)$ by $p_i$.  And since
$\path(y^d_{i+1})$ has $p_i$ as an initial
substring, we will write it as the concatenation
$p_iq_i$, where $q_i$ is possibly $\Lambda$.

We claim that none of the $p_i$ is an inital
substring of any of the others.  For suppose
$i \neq j$ and $p_i$ is an initial 
substring of $p_j$.  Since $y_j^u$
corresponds to a leaf of $T(\term(y_j^u))$, 
we must have $\term(y_i^u) \neq \term(y_j^u)$.
Now consider $y_{j+1}^d$.  We have that 
$\term(y^d_{j+1}) \neq \term(y_j^u)$, so
$\term(y^d_{j+1}) = \term(y_i^u)$.  We also
have that $p_i$ is an initial substring 
of $p_j$, which is an initial 
substring of $p_jq_j = \path(y_{j+1}^d)$.
In $T(\term(y_i^u))$, this would place the leaf 
corresponding to the occurrence
$y_{j+1}^d$ below the leaf corresponding to
$y_i^u$. The only way this could happen is if
$y_i^u = y_{j+1}^d$.  So $i = j+1$, and 
$y_i^u = y_i^d$.  But then 
$\term(y_{i-1}^u) \neq \term(y_i^d) = 
\term(y_i^u) \neq \term(y_{i+1}^d)$, so
$y_{i-1}^u$ and $y_{i+1}^d$ are occurrences
in the same term.  Now $\path(y_{i-1}^u)$ in an
initial substring of $\path(y_i^d) = \path(y_i^u)$,
which is an initial substring of $\path(y_{i+1}^d)$,
implying that both $\path(y^u_{i-1})$ and $\path(y^d_{i+1}$
label the same leaf of the
tree they are in.  So $i-1 = i+1 = j$, and our cycle
consists of just $y_i$ and $y_j$, with 
$y_i^u = y_i^d$ and $y_j^u = y_j^d$.  This is a 
contradiction, since at least one variable occurrence
in a cycle must be strictly above the next occurrence.
The claimi s established.

Without loss of generality, assume that the occurrence
$y_0^u$ is strictly above $y_1^d$, so $\path(y_1^d)$
is $p_0q_0$ where $q_0 \neq \Lambda$. 
Let $I = \{ 0,1,2,\dots k-1 \}$ be our set of indices for 
the $y_i$, and let $N$ be $\{ i \in I : q_i \neq \Lambda \}$
So $0 \in N$.

Define the relation $\sim$
 on $I$ by
$i \sim j$ iff $j = (i+1) \mod k$ and $q_i = \Lambda$, and 
let $\approx$ be the equivalence relation generated by $\sim$.
Intuitively, the classes of $\approx$ are runs of consecutive
indices, with each class ending at an element of $N$.  

Finally, define $f : I \rightarrow I$ by letting
$f(i)$ the least element of the $\approx$ equivalence class
of $i$.  This gives us that $f(i) = f(i+1)$ when 
$q_i = \Lambda$.  (We usually have $f(i) \neq f(i+1)$
when $q_i \neq \Lambda$.  The one exception is when
only one of the $q_j$ is not $\Lambda$, so 
$i$ and $i+1$ are related by $\approx$ the long way around
the cycle.)

Now we define the groupoid operation $\star$ to be the sum \ 
$[\left\|k+f(0),p_0,0\right\| + \left\|k+f(1),p_1,1\right\| + \cdots
\left\|k + f(k-1),p_{k-1},k-1\right\|] + 
\sum_{i \in N} \left\|k+f(i+1),q_i,k+f(i)\right\|$
\ The groupoid operation
 $\star '$ will be 
 $\star + 
\left\|k+f(1),q_0,k+f(0)\right\|' - \left\|k+f(1),q_0,k+f(0)\right\|$, \ 
a slight variation of $\star$
 where the operation $\left\|k+f(1),q_0,k+f(0)\right\|$ is replaced with
 the tweaked operation $\left\|k+f(1),q_0,k+f(0)\right\|'$, while
 all of the other operations remain unchanged.  
 
 We will 
 show that in the groupoid with operation $\star$, that
 the sum modulo 2 of $s[0]+s[1]+ \dots s[k-1]$
 will always equal the sum modulo 2 of 
 $t[0]+t[1]+\dots t[k-1]$.  Then we will
 confirm that in the groupoid with operation $\star '$,
 the two corresponding sums of components will differ.
 This difference
 will be caused by the tweaked operation $\left\|k+f(1),q_0,k+f(0)\right\|'$,
 which will only produce an effect in the final output in
 $\term(y_1^d)$, the term where the occurrence $y_1^d$
 lies.  For the moment, we will be working with the 
 operation $\star$.
  
 First, we establish that for any $y_i$, the value of
 the $i$-th component of $\term(y_i^u)$ will be $y_i[k+f(i)]$.
 Without loss of generality, let $\term(y_i^u)$ be $s$.
 The only summand of 
  $\star$ that assigns a value to
 $s[i]$ is $\left\|k+f(i),p_i,i\right\|$, so $s[i]$ will have the value
 it assigns.  We apply Lemma \ref{path lemma}, and get 
 that $s[i]$ is equal to $r[k+f(i)]$, where $r$ is the 
 subterm of $s$ with path $p_i$.  In this case, 
 $r = y_i$, so $s[i] = r[k+f(i)] = y_i[k+f(i)]$, as 
 desired.
 
 Given any $i$, we let $j = i+1 \mod k$.   
 We now show that for any $y_j$, the value 
 of the $i$-th component of $\term(y_j^d)$ is also 
 $y_j[k+f(j)]$.  Without loss of generality, let 
 $\term(y_j^d)$ be $t$.  As before, $t[i]$ will be 
 equal to $r[k+f(i)]$, where $r$ is the subterm of 
 $t$ with path $p_i$.  We now have two cases.  If $i \notin N$,
 then $q_i = \Lambda$ and $\path(y_j^d) = p_i$, making $r = y_j$
 and $t[i] = y_j[k+f(i)] = y_j[k+f(j)]$ since $i \sim j$.
 So assume $i \in N$.  Then $r$ is a nontrivial subterm of $t$,
 where $\path_r(y_j^d)$ is $q_i$.  The only term in $\star$ that
 assigns a value to $r[k+f(i)]$ is $\left\|k+f(i+1),q_i,k+f(i)\right\|$, so 
 $r[k+f(i)]$ is $y_j[k+f(i+1)] = y_j[k+f(j)]$, as desired.
  
 For each $i$, we do not know which of $s$ and $t$ the 
 occurrences $y_i^u$ and $y_{i+1}^d$ are in.  This turns out
 not to be an obstacle, since we do know that $y_i^u$ and 
 $y_{i+1}^d$ occur in different terms.  Working modulo $2$, we have that 
 $s[0] + s[1] + \cdots s[k-1] + t[0] + t[1] + \cdots t[k-1]
 = [y_0[k+f(0)] + y_1[k+f(1)] + \cdots y_{k-1}[k+f(k-1)]] + 
 [y_1[k+f(1)]+ y_2[k+f(2)]+ \cdots y_k[k+f(k)]]$,
 where the second group on the right hand side comes from the
 $y_j^d$.  But the latter expression is equal to
 $2 [y_0[k+f(0)] + \cdots y_{k-1}[k+f(k-1)]] = 0$ modulo $2$.  
 Since $s[0] + \cdots s[k-1]$ and $t[0] + \cdots t[k-1]$ sum
 to $0$, they have the same parity.
 
 Now we turn to the groupoid with operation $\star '$, and
 consider the effect of the tweaked operation $\left\|k+f(1),q_0,k+f(0)\right\|'$.
 The reader may verify that everything works as before,
 except in the calculation of the $0$-th component of 
 $\term(y_1^d)$.  As before, we may assume that $s$ is
 $\term(y_1^d)$.  We then get $s[0] = r[k+f(0)]$, where
 $r[k+f(0)]$ is found using $\left\|k+f(1),q_0,k+f(0)\right\|'$.  This makes
 $r[k] = y_1[k+f(1)]+1 \mod 2$, giving $s[0] = y_1[k+f(1)]+1 \mod 2$.
 This in turn changes the parity of $s[0] + s[1] + \cdots s[k-1]$
 in whichever term we are calling $s$, as desired.
 
As in Theorem \ref{cover theorem}, this yields a finite groupoid
that separates $s$ and $t$.  
\end{proof}
 
We would like to have a nice characterization of which pairs
of groupoid terms can be separated in a finite groupoid.  So 
we will also investigate when it is impossible to separate a
pair of terms in any groupoid.  

We need a bit of preliminary material on free algebras.  A more
detailed exposition may be found in \cite{BurrisSankappanavar}.
We use $\mathbf G$ for the class of all groupoids, and let
$F_{\mathbf G}(y_0, y_1, \ldots y_{n-1})$ denote the free groupoid
with generators $y_0, y_1, \ldots y_{n-1}$.  The key feature of
$F_{\mathbf G}(y_0, y_1, \ldots y_{n-1})$ is that it has the 
Universal Mapping Property for the class of groupoids.  If $G$
is any groupoid with elements 
$g_0, g_1, \ldots g_{n-1}$, then there is a unique homomorphism 
$\phi$ from $F_{\mathbf G}(y_0, y_1, \ldots y_{n-1})$ into
$G$ where $\phi(y_i) = g_i$ for all $i$.

\begin{theorem}\label{TFAE theorem}
Let $s$ and $t$ be groupoid terms, each on a set of variables
that is a subset of $\{ y_0, y_1, \ldots y_{n-1} \}$.  
Then the following are equivalent.
\begin{enumerate}
	\item  $s$ and $t$ are separated in some groupoid.
	\item  $s$ and $t$ are separated in 
	$F_{\mathbf G}(y_0, y_1, \ldots y_{n-1})$.
	\item  $s$ and $t$ are separated in $F_{\mathbf G}(x)$, the free
	groupoid on one variable.
\end{enumerate}
\end{theorem}

\begin{proof}
Let $s$ and $t$ be groupoid terms with all their variables
in $\{ y_0, y_1, \ldots y_{n-1} \}$.  It is clear that 
(2) implies (1).  To see that (3) implies (2), suppose that
(2) fails.  Then there are terms $h_0, h_1, \ldots h_{n-1}$ in
$F_{\mathbf G}(y_0, y_1, \ldots y_{n-1})$ with 
$s(h_0, h_1, \ldots h_{n-1}) = t(h_0, h_1, \ldots h_{n-1})$.
The $h_i$ are all generated from $\{ y_0, \ldots y_{n-1} \}$
by repeatedly using the groupoid operation.  Now consider
the homomorphism $\phi$ from 
$F_{\mathbf G}(y_0, y_1, \ldots y_{n-1})$ into 
$F_{\mathbf G}(x)$ that takes all of the $y_i$ to $x$.
Denoting the image of each $h_i$ by $h'_i$, we have that 
$s(h'_0, h'_1, \ldots h'_{n-1}) = t(h'_0, h'_1, \ldots h'_{n-1})$
in $F_{\mathbf G}(x)$, so (3) fails.

To see (1) implies (3), assume that (3) fails.  So we have
$f_0, f_1 , \ldots f_{n-1} \in F_{\mathbf G}(x)$ with 
$s(f_0, f_1, \ldots f_{n-1}) = t(f_0, f_1, \ldots f_{n-1})$.
Letting $G$ be any groupoid, we pick any $c \in G$, and 
consider the homomorphism $\phi$ from $F_{\mathbf G}(x)$ to 
$G$ that takes $x$ to $c$.  Letting the image of each 
$f_i$ be $f'_i$, we have that 
$s(f'_0, f'_1, \ldots f'_{n-1}) = t(f'_0, f'_1, \ldots f'_{n-1})$
in $G$, so (1) fails.
\end{proof}

The free groupoid $F_{\mathbf G}(x)$ is easy to work with,
since all of its elements may be viewed as groupoid terms
in the single variable $x$.  Terms $s$ and $t$ are separated
in $F_{\mathbf G}(x)$ iff there are no terms 
$f_0(x), f_1(x), \ldots f_{n-1}(x) \in F_{\mathbf G}(x)$ 
that can be substituted for the variables
of $s$ and $t$ to yield 
$s(f_0(x),f_1(x), \ldots f_{n-1}(x)) = 
t(f_0(x), f_1(x), \ldots f_{n-1}(x))$.  

This relates to the notion of {\em unification} of terms, which 
has been extensively studied in computer science.  The introduction
of the topic was by Herbrand, in \cite{Herbrand}. 
Modern work was pioneered by Robinson, in \cite{Robinson}.
Good 
survey articles are by Knight (in \cite{Knight}) and Jouannaud
and Kirchner (in \cite{JK}).  Consider two terms $s(x_0, \dots x_{m-1})$ 
and $t(y_0, \dots y_{n-1})$.  
The terms are {\em unifiable} if there are terms 
$r_0,\dots r_{m-1}$ and $u_0, \dots u_{n-1}$ so that substituting
the $r_i$ for the $x_i$ in $s$ and the $u_j$ for the $y_j$ in $t$
makes the two resulting terms identical, and the corresponding
substitution is a {\em unification}.  In other words, the terms
$s$ and $t$ can be unified iff they can not be separated in a free
algebra.  In view of the previous theorem, two terms can not be unified
iff there is a groupoid where they are separated.

Algorithms to see whether or not two terms $s$ and $t$ can be 
unified are discussed in detail in \cite{Knight} and \cite{JK}.  
An inefficient but effective method for groupoid terms 
is to use the following rules for generating sets of statements, 
starting with the statement $s = t$.  In each rule, 
a,b,c and d are terms, while x and y are variables.
\begin{enumerate}
	\item {\em (Decompose)}  
	From $a \star b = c \star d$ deduce $a = c$ and $b = d$.
	\item {\em (Coalesce)}  
	If we have $x = y$, deduce the results of replacing every $x$ in our set of statements with a $y$.
	\item {\em (Check)}
	From $x = a$, deduce $\mathbf{False}$ if $x$ occurs in the term $a$.
	\item {\em (Eliminate)}
	From $x = a$, deduce the results of replacing every $x$ in our set of 
	statements with the term $a$, provided $x$ does not occur in $a$.
\end{enumerate}
One may simply apply all the rules repeatedly, until no more statements 
are deduced.  If $\mathbf{False}$ is ever deduced, the original terms 
$s$ and $t$ can not be unified.  Otherwise, a unifying set of substitutions
will be deduced. In practice, one may be more targeted in applying the rules
and reach $\mathbf{False}$ or unifying substitutions more rapidly.

For example, consider $s = (x \star y) \star (z \star y)$ and
$t = z \star ((x \star y) \star (x \star x))$.  We will use the 
algorithm to see if they can be unified.  We start with $s = t$.
Using {\em Decompose}, we obtain $z = x \star y$ and
$z \star y = (x \star y) \star (x \star x)$.  Applying {\em Decompose}
again to the last statement, we get $z = x \star y$ (a duplicate) and
$y = x \star x$.  Applying {\em Eliminate} using $y = x \star x$ to
$z = x \star y$, we get $z = x \star (x \star x)$.  We have found a
set of unifying substitutions.  Letting $y = x \star x$ and
$z = x \star (x \star x)$ in $s$ and $t$, both become
$(x \star (x \star x)) \star ((x \star (x \star x)) \star (x \star x))$.

Here is an example with a cycle.  In view of Theorem \ref{extended theorem},
it will be no surprise that this is an obstacle to unification.
Let $s = (x \star y) \star (z \star w)$ and let 
$t = ((w \star u) \star x) \star ((y \star v) \star z)$.
Repeatedly applying {\em Decompose}, we get $x = w \star u$,
$y = x$, $z = y \star v$ and $w = z$.  Applying {\em Coalesce},
we get $x = z \star u$ and $z = x \star v$.  Applying {\em Eliminate}
gives $x = (x \star v) \star u$, and applying {\em Check} gives 
$\mathbf{False}$.  Our applications of {\em Coalesce}
and {\em Eliminate} acted to reduce the length of the 
original cycle.  Using the
notation of Theorem \ref{extended theorem}, this cycle had
$y_0^u$ the $x$ in $s$, $y_0^d$ the $w$ in $t$, $y_1^u$ the $w$ in $s$,
$y_1^d$ the $z$ in $t$, $y_2^u$ the $z$ in $s$, $y_2^d$ the $y$ in $t$,
$y_3^u$ the $y$ in $s$, and $y_3^d$ the $x$ in $t$.

There are pairs of terms without a cycle which still can not be unified.
For example, let $s = (x \star y) \star (z \star y)$ and let 
$t = z \star ((y \star y) \star (x \star x))$.  Working left to right, we
see that $x$ and $y$ occur below $z$, $y$ occurs below $z$ and $x$ occurs
below $y$.  This is consistent with the ordering $x < y < z$.
Since there is a consistent ordering of the variables like this, 
there are no cycles.  However, $s$ and $t$ can not be unified.  Applying 
{\em Decompose} repeatedly gives $z = x \star y$, $z = y \star y$ and
$y = x \star x$.  Then applying {\em Eliminate} to the first two
gives $x \star y = y \star y$, after which {\em Decompose} gives 
$x = y$.  Finally, {\em Eliminate} gives $x = x \star x$, and 
{\em Check} gives $\mathbf{False}$.

Although Theorem \ref{extended theorem} does not apply to this last 
example, we had no problem separating the terms using a similar 
construction.  Letting $\star' = \left\|3,l,0\right\| + \left\|3,rl,1\right\| + 
\left\|4,rr,2\right\| + \left\|4,l,3\right\| + \left\|4,l,4\right\|'$, we calculate $s[0] + s[1] + s[2]
= (x \star y)[3] + z[3] + y[4] = x[4] + z[3] + y[4]$, while
$t[0]+t[1]+t[2] = z[3] + (y \star y)[3] + (x \star x)[4] =
z[3] + y[4] + x[4] + 1$, which has the opposite parity.

Based on many examples similar to the above, we make the following 
conjecture.

\begin{conjecture} 
Whenever two groupoid terms can be separated in an
infinite groupoid, they can also be separated in a finite groupoid.
\end{conjecture}

\end{section}



\centerline{\large\bf Addresses}\vspace{.5em}

$^1$Milton Braitt. Departamento de Matem\'{a}tica, 
Universidade Federal de Santa Catarina, Cidade Universit\'{a}ria,
Florian\'{o}polis, SC 88040-900, Brasil

\underline{Email}: \ MSBraitt@mtm.ufsc.br\ 

$^2$David Hobby and Donald Silberger. Department of Mathematics, State University
of New York, New Paltz NY 12561 -- U.S.A.

\underline{Emails}: \ hobbyd@newpaltz.edu \ or \ silbergd@newpaltz.edu\

\vspace{4em}

\noindent 2010 Mathematics Subject Classification: Primary: 20N02, 08A99  Secondary: 68Q99, 68T15

\noindent Keywords: groupoids, non-associative, unification

\end{document}